\documentclass[12pt, reqno]{amsart}
\usepackage[cp1251]{inputenc}
\usepackage{amsmath,amsopn,amssymb,amsthm, wasysym}
\usepackage[breaklinks=true,colorlinks=true,linkcolor=blue,citecolor=blue,urlcolor=blue]{hyperref}
\usepackage[english]{babel}
\usepackage{graphicx}

\renewenvironment{proof}[1][Proof]{\textbf{#1.} }{\ \rule{0.5em}{0.5em}}

\DeclareMathOperator{\Lin}{Lin}

\DeclareMathOperator{\Isom}{Isom}

\DeclareMathOperator{\conv}{conv}

\renewenvironment{proof}[1][Proof]{\textbf{#1.} }
{\ \rule{0.5em}{0.5em}}
\newtheorem{theorem}{Theorem}
\newtheorem{prop}{Proposition}
\newtheorem{lemma}{Lemma}
\newtheorem{corollary}{Corollary}
\newtheorem{question}{Question}
\newtheorem{conjecture}{Conjecture}

\theoremstyle{definition}
\newtheorem{definition}{Definition}
\newtheorem{remark}{Remark}
\newtheorem{example}{Example}

\newtheorem{problem}{Problem}

\begin{document}

\title
[Perfect and almost perfect homogeneous polytopes]
{Perfect and almost perfect homogeneous polytopes}
\author{V.N.~Berestovski\u\i, Yu.G.~Nikonorov}

\address{Berestovski\u\i\  Valeri\u\i\  Nikolaevich \newline
Sobolev Institute of Mathematics of the SB RAS, \newline
4 Acad. Koptyug Ave., Novosibirsk 630090, RUSSIA, \newline
Principal scientific researcher}
\email{valeraberestovskii@gmail.com}

\address{Nikonorov\ Yuri\u\i\  Gennadievich\newline
Southern Mathematical Institute of VSC RAS \newline
53 Vatutina St., Vladikavkaz, 362025, RUSSIA, \newline
Principal scientific researcher}
\email{nikonorov2006@mail.ru}

\begin{abstract}
The paper is devoted to perfect and almost perfect homogeneous polytopes
in Euclidean spaces. We classified perfect and almost perfect polytopes among all regular polytopes
and all semiregular polytopes excepting Archimedean solids and two four-dimensional Gosset polytopes.
Also we construct some non-regular homogeneous polytopes that are (or are not) perfect and
posed some unsolved questions.

\vspace{2mm}
\noindent
2020 Mathematical Subject Classification:
52A20, 52A40, 53A04.

\vspace{2mm} \noindent Key words and phrases: almost perfect polytope, convex polytope, homogeneous polytope, lattice, linear group representation,
L\"{o}wner~---~John ellipsoid, perfect polytope.
\end{abstract}

\maketitle

\hfill {\sl Dedicated to Professor Anatoly Georgievich Kusraev}

\hfill {\sl on the occasion of his 70th birthday}

\section{Introduction}\label{sect.0}

A finite metric space $(M,d)$ is homogeneous if its isometry group $\Isom(M)$ of $(M,d)$ acts transitively on $M$.
A finite homogeneous metric subspace of an Euclidean space
represents the vertex set of a compact convex polytope
with the isometry group that is transitive on the vertex set;
in each case, all vertices lie on a sphere.
In \cite{BerNik19, BerNik21, BerNik21n}, the authors  obtained the complete description of the metric properties
of the vertex sets of regular and semiregular polytopes
in Euclidean spaces from the point of view of the normal homogeneity and the Clifford~---~Wolf homogeneity, see also the survey~\cite{BerNik21nn} for additional results
on polytopes with normal homogeneous or Clifford~---~Wolf homogeneous vertex sets and \cite{BerNik20}
for the corresponding results on normal homogeneous and Clifford~---~Wolf homogeneous Riemannian manifolds.
Some properties of $m$-point homogeneous finite subspaces of Euclidean spaces were discussed in \cite{BerNik22}.

We consider Euclidean space $\mathbb {R}^n$ (with the standard inner product $(x,y)=\sum_{i=1}^n x_i y_i$) with the origin $O$
and the unit hypersphere $S:=S(O,1) \subset \mathbb{R}^n$.
In what follows, we will use the corresponding Euclidean norm $\|\cdot \|=\sqrt{(\cdot, \cdot)}$.

Recall that a polytope~$P$ in~$\mathbb{R}^n$
is {\it homogeneous\/} (or {\it vertex-transitive}) if
its symmetry (isometry) group acts transitively on the set of its vertices.
Hence, a convex polytope $P \in \mathbb{R}^n$ is homogeneous if and only if its vertex set with the metric $d$, induced by the Euclidean metric in  $\mathbb{R}^n$,
is homogeneous.

Since the barycenter of a finite system of material points (with one and the same mass) in any Euclidean space is preserved for any bijection
(in particular, any isometry)
of this system, we have the following result.

\begin{prop}[\cite{BerNik19}] \label{pr.efhs}
Let $M=\{x_1, \dots, x_q\}$, $q\geq n+1$, be a finite homogeneous metric subspace of Euclidean
space $\mathbb{R}^n$, $n\geq 2$. Then
$M$ is the vertex set of  a convex polytope~$P$, that is situated in some sphere in $\mathbb{R}^n$ with radius
$r>0$ and center $x_0=\frac{1}{q}\cdot\sum_{k=1}^{q}x_k$.
In particular, $\Isom(M,d)\subset O(n)$.
\end{prop}

This result shows that {\it the theory of convex polytopes} is very important for the study of finite homogeneous subspaces of Euclidean spaces,
see \cite{BoFe1987, Cox73, FToth64}.

Let us recall the definition of {\it almost perfect polytope},
that appeared for the first time in \cite{BerNik21n}.

\begin{definition}\label{de.alper}
We will call a nondegenerate homogeneous polytope $P \subset \mathbb{R}^n$ {\it almost perfect},
if every second order hypersurface in $\mathbb {R}^n$, containing all vertices of $P$ and symmetric with respect to the center of $P$,
coincides with the circumscribed hypersphere of $P$.
\end{definition}

\begin{definition}\label{de.per}
We will call a nondegenerate homogeneous polytope $P \subset \mathbb{R}^n$ {\it perfect},
if every second order hypersurface in $\mathbb {R}^n$, containing all vertices of $P$,
coincides with the circumscribed hypersphere of $P$.
\end{definition}

It is clear that every perfect homogeneous polytope is almost perfect.

\smallskip

\begin{lemma}[\cite{BerNik21n}]\label{le.perf_0}
A homogeneous polytope $P \subset \mathbb{R}^n$ is {\rm(}almost{\rm)} perfect if and only if
any ellipsoid $E$ in $\mathbb {R}^n$ {\rm(}with center coinciding with the center of $P${\rm)}, which includes all vertices of $P$, coincides with the circumscribed hypersphere of $P$.
\end{lemma}

It could be explained as follows. Without loss of generality, we may suppose that the origin $O$ is the center of a homogeneous polytope $P$. Then all vertices of $P$ satisfy
the condition $\|x\|^2-r^2=0$, where $r$ is the radius of the curcumscribed hypersphere for $P$. If we have a quadric $F(x)=0$ containing all vertices of $P$ and distinct
from the above hypersphere, then all vertices of $P$ are situated on any quadric of the form $\|x\|^2-r^2 +\varepsilon \cdot F(x)=0$, where $\varepsilon \in \mathbb{R}$.
It is clear that this quadric is bounded (hence, is an ellipsoid) for sufficiently small $|\varepsilon|>0$. Moreover, if the origin $O$ is the center for  quadric $F(x)=0$,
than it is the center for all quadrics from the above one-parameter family.

\begin{lemma}\label{le.sym}
Let $P$ be a non-degenerate homogeneous {\rm(}bounded{\rm)} polytope in $\mathbb{R}^n$, $n\geq 2$, which is centrally symmetric with respect to
the origin $O\in\mathbb{R}^n$. Then every second order hypersurface in $\mathbb{R}^n$, containing all vertices in $P$, is symmetric
with respect to the center $O$.
\end{lemma}

\begin{proof}
Assume that a second order hypersurface in $\mathbb{R}^n,$ containing all vertices in $P,$ is defined by an equation
\begin{equation}\label{eq_cubepol1}
F(x)=\sum_{i,j=1}^n a_{ij} x_ix_j+\sum_{i=1}^n b_i x_i+ c=0,
\end{equation}
where $a_{ij},b_i,c \in \mathbb{R}$, and $a_{ij}=a_{ji}$ for all indices.

Let $v=(v_1,\dots,v_n)$ be any vertex of $P$. Then $-v$ is also a vertex of $P$ and
$0=F(v)-F(-v)= 2\Sigma_{i=1}^nb_iv_i$, so the vector $(b_1,\dots, b_n)$ is orthogonal to the
vector $v$. Since every vector in $\mathbb{R}^n$ can be represented as a linear combination of some vertices $v$ of $P$,
then $b_1=\dots=b_n=0$. Hence the surface defined by the equation (\ref{eq_cubepol1}) is symmetric with respect to the origin $O$.
\end{proof}

\begin{corollary}\label{cor.csym}
If a homogeneous polytope $P$ in $\mathbb{R}^n$, $n\geq 2$, is almost perfect and centrally symmetric with respect to its center, then
$P$ is perfect.
\end{corollary}

Apart from the fact that the properties of homogeneous polytopes being perfect or almost perfect are of interest in their own right,
they provide some tools for studying other properties of homogeneous polytopes.
For instance, rather restrictive conditions are established for the structure of Clifford~---~Wolf translations on the vertex sets of almost perfect polytopes,
see Proposition~11 in \cite{BerNik21n}.
\smallskip

Important examples of homogeneous polytopes are regular and semiregular polytopes.

Any one-dimensional polytope is a closed segment,
bounded by two endpoints.
It is regular by definition.
Two-dimensional regular polytopes are regular
polygons on Euclidean plane.
For other dimensions, regular polytopes are defined inductively.
A convex $n$-dimensional polytope for $n \geq 3$ is called {\it regular},
if it is homogeneous
and all its facets are regular polytopes congruent to each other.
This definition is equivalent to other definitions of regular convex
polytopes.

We also recall the definition of semiregular convex polytopes.
For $n=1$ and $n=2$, semiregular polytopes are defined as regular.
A convex $n$-dimensional polytope for $n\geq 3$ is called
{\it semiregular} if it is homogeneous
and all its facets are regular polytopes.
A more detailed exposition of regular and semiregular polytopes in Euclidean spaces can be found in various sources, for example, in \cite{Cox73}.

The following problem is very natural.

\begin{problem}\label{prob.1}
Find all perfect and all almost perfect polytopes among regular and semiregular polytops.
\end{problem}

The main result of this paper is as follows.

\begin{theorem}\label{th.main.1}
Let $P$ be a regular polytope in Euclidean space $\mathbb{R}^n$, $n\geq 2$. Then the following assertions hold.

{\rm  1.} $P$ is perfect exactly in the following cases: a regular $m$-gon in $\mathbb{R}^2$ for $m\geq 5$; the dodecahedron and the icosahedron in $\mathbb{R}^3$;
the $24$-cell, the $120$-cell, and $600$-cell in $\mathbb{R}^4$.

{\rm  2.} $P$ is not perfect but is almost perfect if and only if it is a regular triangle in $\mathbb{R}^2$.

{\rm  3.} $P$ is not almost perfect exactly in the following cases: the $n$-dimensional cube and the $n$-dimensional orthoplex  for $n\geq 2$; the $n$-dimensional regular
simplex for $n\geq 3$.
\end{theorem}
\smallskip

The proof of this theorem follows immediately from the well known classification of regular polytopes,
Examples \ref{ex.partop} and \ref{ex.orthopl1}, Propositions \ref{pr.regpolyg1} and \ref{pr.regsimpl1},  Theorems \ref{th.regpolyh1} and \ref{th.fourdim1} below.
Recall also that the $2$-dimensional cube, as well as the $2$-dimensional orthoplex, is the square.
\smallskip

As for semiregular polytopes, we have not yet  thoroughly studied Archimedean solids and two $4$-dimensional
Gosset polytopes and we don't know yet are they
perfect or almost perfect.
We hope that a complete classification of perfect and almost perfect semiregular
polytopes will lead
 to the discovery of new interesting effects and will make it possible to advance in solving related problems.
\smallskip

The paper is organized as follows. In Section \ref{sect.1}
we consider some natural examples of perfect and almost perfect homogeneous polytope, as well as examples of polytopes that have no such properties.
In particular, we give a comprehensive description of all regular polygons from this point of view in Proposition \ref{pr.regpolyg1}.

Section \ref{sect.2} is devoted to the study of various relationships between
the reducibility of the isometry group of polytopes, their L\"{o}wner~---~John ellidsoids and the (almost) perfectness.

In the last Section \ref{sect.3}, we provide more complicated
examples of perfect polytopes, as well as some partial classifications.
In particular, we obtain a description of all (almost) perfect $3$-dimensional and $4$-dimensional regular polytopes in Theorems \ref{th.regpolyh1} and \ref{th.fourdim1}.
and we prove that the Gosset semiregular polytopes in dimensions $6$, $7$, and $8$ are perfect.
\smallskip

To obtain the results, we used different facilities connecting to corresponding polytopes, including the natural linear representations in ambient
Euclidean spaces of their isometry subgroups, their L\"{o}wner~---~John ellidsoids, quadrics, cubes of the greatest volume inscribed into them,
distance spheres in their vertex sets with vertices as centers, and the lattice packing in $\mathbb{R}^8$ for the lattice $E_8$.
\medskip

{\bf Acknowledgments.} The work of the first author was carried
out within the framework of the state Contract to the IM SB RAS, project FWNF-2022-0006.

\section{Some natural examples}\label{sect.1}

In this section, we collect some examples of perfect and almost perfect polytopes in Euclidean spaces, as well as
the examples of polytopes that have not these properties.

\begin{lemma}\label{le.four}
The intersection of a given circle $S$ on the Euclidean plane with any curve~$Q$ of the second order, distinct from $S$,
contains at most $4$ points.
\end{lemma}

\begin{proof}
Without loss of generality, we may assume that $S=\{(x,y)\,|\, x^2+y^2=1\}$ and $(-1,0) \not\in Q$.
Then we can parameterize $S\setminus \{(-1,0)\}$ as follows: $(x,y)=\left(\frac{1-t^2}{1+t^2},\frac{2t}{1+t^2}\right)$, $t\in \mathbb{R}$.
The curve $Q$ is determined by an equation $F(x,y)=0$, where $F(x,y)=a_{11}x^2+2a_{12}xy+a_{22}y^2+b_1x+b_2y+c$ for some $a_{11},a_{12},a_{22},b_1,b_2,c \in \mathbb{R}$.
Using the above parametrization, we get the equation $(1+t^2)^2\cdot F\left(\frac{1-t^2}{1+t^2},\frac{2t}{1+t^2}\right)=0$ for points of $S\cap Q$.
Obviously, this algebraic equation of the fourth degree has no more than 4 (real) solutions.
\end{proof}

We recall that all regular and semiregular convex polytopes in Euclidean spaces are homogeneous, see details in \cite{BerNik21n, BerNik21nn}.
At first, we consider regular polygons.

\begin{prop}\label{pr.regpolyg1}
Let $P$ be a regular $m$-gon in Euclidean plane $\mathbb{R}^2$. Then the following assertions hold:

{\rm  1.} $P$ is almost perfect, but is not perfect for $m=3$;

{\rm  2.} $P$ is not almost perfect for $m=4$;

{\rm  3.} $P$ is perfect for any $m\geq 5$.
\end{prop}

\begin{proof}
The last assertion follows from Lemma \ref{le.four}.

Let us consider the case $m=4$. Take any $L\in \mathbb{R}$, then the second-order curve $Lx^2+(2-L)y^2=1$ contains four points of the form
$\left(\pm \frac{1}{\sqrt{2}},\pm \frac{1}{\sqrt{2}}\right)\in \mathbb{R}^2$, i.~e.
vertices of a square with side length $\sqrt{2}$. The circumscribed circle around this square is obtained at $L=1$.
Hence, the square is not almost perfect.

On the other hand, a regular triangle is almost perfect. Indeed, any second-order curve containing the vertices of a regular triangle
and symmetric with respect to its center $O$ also contains points symmetric to its vertices with respect to $O$, i.~e.
it contains the vertices of a regular hexagon, and therefore coincides with the circumscribed circle.
In order to prove that a triangle is not perfect, it suffices to note that the vertices $(0,-2\sqrt{3})$, $(-3,\sqrt{3})$, and $(3,\sqrt{3})$
are situated both on the circle $x^2+y^2=12$ and on the parabola $y=x^2/\sqrt{3}-2\sqrt{3}$.
\end{proof}
\smallskip

\begin{example}\label{ex.spectetr}
Let us consider the convex hull $P$ of the points $B_i\in \mathbb{R}^3$, $1\leq i \leq 4$, where
$$
B_1=(1,0,0),\, B_2=(0,1,0),\, B_3=\left(-\frac{1}{2}, -\frac{1}{2}, \frac{\sqrt{2}}{2}\right),\,
B_4=\left(-\frac{1}{2}, -\frac{1}{2}, -\frac{\sqrt{2}}{2}\right).
$$
It is easy to see that $P$ is a homogeneous tetrahedra,
$\|B_1\|=\|B_2\|=\|B_3\|=\|B_4\|=1$ (hence, all vertices of $P$ are on the  sphere $S$ with center at the origin and unit radius),
$\|B_1-B_2\|=\|B_3-B_4\|=\sqrt{2}$, and
$\|B_1-B_3\|=\|B_1-B_4\|=\|B_2-B_3\|=\|B_2-B_4\|=\sqrt{3}$.

Let us find all quadrics in $\mathbb{R}^3$, that contains all points $B_i$ for $i=1,2,3,4$ (recall that $B_i \in S$).
An arbitrary quadric $E$ is determined by the equation
$$
a_{11}x^2+a_{22}y^2+a_{33}z^2+2a_{12}xy+2a_{13}xz+2a_{23}yz+b_1x+b_2y+b_3z+c=0,
$$
where $(x,y,z)\in \mathbb{R}^3$, $a_{ij}$, $b_i$, $i,j=1,2,3$, and $c$ are some real numbers.
The condition $B_i \in E$, $i=1,2,3,4$,  is equivalent to the following one:
$$
b_1=-a_{11}-c, \quad b_2=-a_{22}-c, \quad b_3=a_{13}+a_{23},\quad a_{12}=-\frac{3}{2}\, a_{11}-\frac{3}{2}\, a_{22}-a_{33}-4c.
$$
Hence, we get a $6$-parameter family of quadrics containing all vertices of $P$. In particular, the polytope $P$ is not almost perfect
(since we can take $a_{13}=-a_{23}\neq 0$ and $a_{11}=a_{22}=-c$).
\end{example}

\begin{example}\label{ex.partop} Let $\psi_i$ be the reflection with respect to the hyperplane $x_i=0$ in $\mathbb{R}^n$, $i=1,\dots,n$.
The group $G=\left(\mathbb{Z}_2\right)^n$ generated by these reflections is transitive on the set $B$ consisting of the points
$(\pm b_1, \pm b_2, \dots, \pm b_n)$, where $b_i>0$, $i=1,\dots,n$.
It is clear that $P:=\conv(B)$ is a (rectangular) parallelotope. Let us consider an ellipsoid $E$ with the equation
$$
\frac{x_1^2}{a_1^2}+\frac{x_2^2}{a_2^2}+\frac{x_3^2}{a_3^2}+\cdots +\frac{x_n^2}{a_n^2}=1,
$$
where $a_i>0$, $i=1,\dots,n$, are some constant. The condition $B\subset E$ is equivalent to the following one:
$$
\sum_{i=1}^n\frac{b_i^2}{a_i^2}=1.
$$
This gives us a $(n-1)$-parameter family of ellipsoids containing the set $B$. In particular, the polytope $P$ is not almost perfect for $n\geq 2$.
If $b_1=b_2=\cdots=b_n$, we get a $n$-dimensional hypercube.
\end{example}

\begin{example}\label{ex.hpartop}
Let us consider the subset $\widetilde{B}$ of the set $B$ from the previous example, consisting of all points
$(\pm b_1, \pm b_2, \dots, \pm b_n)$ with even numbers of minuses.
This subset is homogeneous.
Indeed a subgroup $\widetilde{G}$ of the group
$G=\left(\mathbb{Z}_2\right)^n$, containing the isometries with even numbers of reflections, is transitive on the set $\widetilde{B}$.
It is clear that $\widetilde{B}$ is in the same ellipsoids, that contain $B$ in the previous example.
In particular, the polytope $P=\conv \left(\widetilde{B}\right)$ is not almost perfect for $n\geq 2$.
For $n=3$ such polytopes are simplices (if, in addition, $b_1=b_2=b_3$, then we get a regular $3$-simplex).
If $b_1=b_2=\cdots=b_n$, we get a $n$-dimensional demihypercube.
\end{example}

We know that a regular triangle is almost perfect by Proposition \ref{pr.regpolyg1}. Hence, the following question is very natural.

\begin{question}[\cite{BerNik21n}]\label{qu_tetr}
Is an $n$-dimensional regular simplex an almost perfect polytope for $n\geq 3$?
\end{question}

The following proposition gives  a negative answer to this question.

\begin{prop}\label{pr.regsimpl1}
The $n$-dimensional regular simplex is not almost perfect for $n \geq 3$.
\end{prop}

\begin{proof}
The case $n=3$ follows from Example \ref{ex.hpartop}. Indeed, the simplex with the vertices $(1,1,1)$, $(1,-1-1)$, $(-1,1,-1)$, $(-1,-1,1)$ is such that
any quadric of the form
$\alpha x^2+\beta y^2+\gamma z^2=1$,
where $\alpha,\beta,\gamma \in \mathbb{R}$, $\alpha+\beta+\gamma=1$, contains all its vertices. Hence, the $3$-dimensional regular simplex $T_3$ is not almost perfect.

Suppose that $k$-dimensional regular simplex $T_k$ is not almost perfect and prove that $T_{k+1}$ have the same property.
Me may suppose that all sides of $T_k$ have length $1$, the origin $O$ in $\mathbb{R}^k$ is the center of $T_k$, the quadric
$f(x)=0$ with center $O$ contains all vertices of $T_k$ and is distinct from the circumscribed sphere of $T_k$. Moreover, we may suppose in addition that $f(O)\neq 0$.
Indeed, if $f(O)=0$, and $r$ is the radius of the circumscribed sphere of $T_k$,
then we can replace $f(x)$ with $f(x)+\varepsilon \bigl(\|x\|^2-r^2\bigr)$ for any $\varepsilon >0$.

We will use the embedding  $\mathbb{R}^k\ni x \mapsto (x,0)\in \mathbb{R}^{k+1}$. Let us fix a positive number $\mu \in \mathbb{R}$.
For any vertex $A_i$, $i=1,\dots,k+1$, of $T_k$, we consider the point $B_i=(A_i,-\mu)$. Consider also the point $B_{k+2}=(0\dots,0,(k+1)\mu)$.
Obviously the baricenter of the polytope in $\mathbb{R}^{k+1}$ with the vertices $B_i$, $i=1,\dots,k+2$, is the origin $O$. Moreover, there is
(exactly one) $\mu >0$ such that
this polytope is $(k+1)$-dimensional regular simplex. Let us fix this special value of $\mu$.
Now, let us consider the quadric $F(x)=0$ in $\mathbb{R}^{k+1}$, where
$$
F(x_1,\dots,x_k,x_{k+1})= \frac{x_{k+1}^2}{\mu^2}-\frac{k(k+2)}{C}f(x_1,\dots,x_k)-1, \quad C=f(0,0,\dots, 0)\neq 0.
$$
Obviously, that this quadric is not a sphere, but the origin is the center of this quadric. It is easy to check also that $F(B_i)=0$ for all $i=1,\dots,k+2$.
Hence, the proposition is proved by induction.
\end{proof}

\begin{example}\label{ex.orthopl1}
Let us fix a number $d >0$ and consider in $\mathbb{R}^n$ the set of all points, every of which has exactly one coordinate $\pm d$ while other its $n-1$ coordinates are zero. The convex hull of this set is $n$-dimensional orthoplex $P$. Let us find all quadrics in $\mathbb{R}^n$ that contain all vertices of $P$.
Any such quadric has the form $F(x)=0$, where
\begin{equation}\label{eq_orthopol1}
F(x)=\sum_{i,j=1}^n a_{ij} x_ix_j+\sum_{i=1}^n b_i x_i+ c,
\end{equation}
where $a_{ij},b_i,c \in \mathbb{R}$, and $a_{ij}=a_{ji}$ for all indices. If we compare all pair of vertices symmetric each to other with respect to the origin $O$,
we easily get that $b_i=0$ and $a_{ii}d^2+c=0$ for all $1\leq 1 \leq n$. Therefore, $a_{11}=a_{22}=\cdots=a_{nn}= -c\cdot d^{-2}$.
The value of $a_{ij}$ for $i\neq j$ are not important. Therefore, we have the family of quadrics, depending on $n(n-1)/2+1$ parameters.
In particular, the $n$-dimensional orthoplex is not almost perfect for $n\geq 2$. Recall, that we get the square for $n=2$ and the octahedron for $n=3$.
\end{example}

\section{The reducibility,  L\"{o}wner~---~John ellipsoids, and the almost perfectness}\label{sect.2}

Let us recall that any regular $n$-simplex $T_n$ is not perfect for $n\geq 2$ and is not almost perfect for $n\geq 3$ by Proposition \ref{pr.regsimpl1}.
The symmetry group of $T_n$ is the symmetric group $S_{n+1}$, that acts irreducibly on $\mathbb{R}^n$.
Note that $S_3$ has only one proper subgroup (the cyclic group of order $3$),
acting transitively of the vertex set of the triangle $T_2$.

On the other hand, the almost perfectness of a given polytope implies
the non-reducibility of any transitive subgroup of its symmetry group.
\smallskip

Let us consider a finite orthogonal group $\Gamma \subset O(n)$ that have orbits, whose convex hulls are not degenerate.
We denote by $O_{\Gamma}(u)$ the orbit of the group $\Gamma$ through the point $u\in \mathbb{R}^n$.

\begin{prop}\label{pr.orap.2}
If there is a non-zero $u \in \mathbb {R}^n$ such the polytope  $P:=\conv(O_{\Gamma}(u))$ is degenerate, then $\Gamma$ acts reducibly on $\mathbb{R}^n$.
\end{prop}

\begin{proof}
It is clear that $P$ and $\Lin(P)$ is invariant under the action of the group $\Gamma$. Since $P$ is degenerate, then  we have $1 \leq \dim(\Lin(P)) <n$. Hence,
$\Gamma$ acts reducibly on~$\mathbb{R}^n$.
\end{proof}

We are going to study some conditions on $\Gamma$ which imply the property to be (not to be) almost perfect for the convex hulls of the orbits of this group.

\begin{prop}\label{pr.orap.1}
If $\Gamma$ acts reducibly on $\mathbb{R}^n$, then the convex hull of any orbit of $\Gamma$,
that is not situated in some hyperplane, is not almost perfect.
\end{prop}

\begin{proof}
Let us consider the decomposition $\mathbb {R}^n=V_1\oplus \cdots \oplus V_s$ of $\mathbb {R}^n$ as an orthogonal sum of subspaces $V_i$, where $\Gamma$ acts irreducibly.
For any $x\in \mathbb {R}^n$ we consider the corresponding sum $x=\sum_{i=1}^s x_s$, where $x_i \in V_i$.
Let us fix a non-zero $u \in \mathbb {R}^n$, then $\|a(u)\|=\|u\|$ for every $a\in \Gamma$. It is easy to see also that $a(u)=\sum_{i=1}^s a(u_i)$.
Since $\Gamma$ acts irreducibly on every $V_i$, we have $a(u_i)=(a(u))_i$ and $\|u_i\|=\|a(u_i)\|$ for every $a\in \Gamma$ (recall that $\Gamma \subset O(n)$).
Put $r_i:=\|u_i\|$ and $d_i=\dim(V_i)$, then
we get $\|(a(u))_i\|=\|a(u_i)\|=\|u_i\|=r_i$ for any
$a\in \Gamma$ and any $1\leq i \leq s$. Therefore, $\sum_{i=1}^s r_i^2=\sum_{i=1}^s \|(a(u))_i\|^2$ for any $a\in \Gamma$.

Now, let us take some numbers $\alpha_i>0$, $1\leq i \leq s$. Then $\sum_{i=1}^s \alpha_i \|(a(u))_i\|^2=\sum_{i=1}^s \alpha_i r_i^2$ for any $a\in \Gamma$,
i.~e.  we have
$\sum_{i=1}^s \alpha_i \|v_i\|^2=\sum_{i=1}^s \alpha_i r_i^2$ for any $v$ in the orbit $O_{\Gamma}(u)$ of the group $\Gamma$ through the point $u$.
Hence, we have found an $s$-parameter family of ellipsoids, each of which contains all
points of $O_{\Gamma}(u)$. Hence, the polytope $\conv(O_{\Gamma}(u))$ is not almost perfect.
\end{proof}

\begin{remark}
If $P$ is the convex hull of the orbit of a point $u \in \mathbb {R}^n$ under the action of $\Gamma$ in the above proposition,
then $\Gamma$ is a subgroup of the symmetry group of $P$, that is transitive on the vertex set. On the other hand, the symmetry group of $P$ could be
more extensive than $\Gamma$. For example, the convex hull of the orbit of the point $(1,1,\dots,1) \in \mathbb{R}^n$ under the action of the group $\mathbb{Z}_2^n$,
generated by the reflections in all coordinate hyperplanes, is the standard $n$-dimensional cube. Its symmetry group is $\mathbb{Z}_2^n \rtimes S_n$.
\end{remark}

\begin{prop}
\label{pr.comm}
If there exists a commutative isometry subgroup $\Gamma$ of the polytope $P$ in $\mathbb{R}^n$, $n\geq 3$, which is transitive on the vertex set of $P$,
then $P$ is not almost perfect.
\end{prop}

\begin{proof}
It follows from Proposition~\ref{pr.orap.1} and known facts that every real linear representation of a finite group $\Gamma$ is totally reducible
(Corollary in p.~2.3 in \cite{Vinb85}) and is one-dimensional or two-dimensional if $\Gamma$ is commutative and the representation is irreducible
(Exercise~7 at the end of Ch.~1 in \cite{Vinb85}).
\end{proof}

\begin{prop}
\label{pr.notap}
Every homogeneous right prism {\rm(}in particular, cube{\rm)}, in $\mathbb{R}^n$, $n\geq 2$, every right antiprism in $\mathbb{R}^3$,
as well as the regular simplex, orthoplex,
or demihypercube in $\mathbb{R}^n$, $n\geq 3$, are not almost perfect.
\end{prop}

\begin{proof}
Let $P$ be either a right prism, or a right antiprism, or the regular orthoplex.
Then there are two hyperfaces which include all vertices of $P$.
Moreover, these hyperfaces are parallel to each other.
Then in any case there is an isometry subgroup of $P$
which preserves the union of these two hyperfaces and transitive
on the vertex set of~$P$.
This implies that hypotheses of Proposition~\ref{pr.orap.1} are satisfied, hence,
all mentioned polytopes are not almost perfect.

For the regular simplex or demihypercube in $\mathbb{R}^n$, $n\geq 3$, there exists a commutative isometry subgroup $\Gamma$
which is transitive on the vertex set (a cyclic group of order $n+1$ in the case of the simplex).
Then by Proposition~\ref{pr.comm}, both the regular simplex and the demihypercube are not almost perfect.
\end{proof}

\begin{remark}
In fact, any polytope in Proposition~\ref{pr.notap}, with possible exclusion of right prisms in $\mathbb{R}^n$, $n\geq 4$,
(see e.~g. the prisms from Proposition~\ref{pr.prism}), admits a commutative isometry subgroup, transitive on its vertex set.
\end{remark}

\begin{remark}
The homogeneous simplex from Example \ref{ex.spectetr} admits a commutative isometry subgroup $\Gamma$, transitive on its vertex set.
This group $\Gamma$ is generated by two commuting involutive isometries $g_1$ and $g_2,$ where $g_1=(B_1B_2)(B_3B_4)$,
$g_2=(B_1B_3)(B_2B_4)$.
\end{remark}

\begin{question}
Is it true that if a homogeneous polytope $P$ in $\mathbb{R}^n,$ $n\geq 3,$ has no reducible isometry subgroup, transitive on its vertex set, then $P$ is almost perfect?
\end{question}

Since any Archimedean semiregular solid in $\mathbb{R}^3$ admits no reducible isometry group which is transitive on its vertex set and has
``sufficiently many'' vertices, we suggest that the following conjecture is true.

\begin{conjecture}
Any Archimedean solid is perfect.
\end{conjecture}

\smallskip

Recall that for a given convex body (bounded closed subset) $A$ in $\mathbb {R}^n$, {\it its L\"{o}wner~---~John ellipsoid} $E$
is a unique ellipsoid of minimal volume containing $A$, see e.~g. \cite{Ball97, BarBl05}.
It is obvious that $E$ is invariant under every isometry of $P$. In particular, it is easy to check that the L\"{o}wner~---~John ellipsoid
for the hypercube is the circumscribed hypersphere.

\begin{prop}\label{pr.orap.3}
Let $P$ be a non-degenerate homogeneous polytope in $\mathbb{R}^n$. Then either the  L\"{o}wner~---~John ellipsoid
for $P$ coincides with the circumscribed sphere of $P$, or the symmetry group of $P$  acts reducibly on~$\mathbb{R}^n$.
\end{prop}

\begin{proof}
Without loss of generality we may suppose that the
circumscribed sphere of $P$ is $S(O, 1)\subset\mathbb {R}^n$.
We denote by $E$ and $\Gamma$ the L\"{o}wner~---~John ellipsoid and the symmetry group of $P$ respectively.

Let $O_1$ be the center of $E$. Since the L\"{o}wner~---~John ellipsoid for $P$ is unique, then $\gamma(E)=E$ for any $\gamma \in \Gamma$, hence
$\gamma(O_1)=O_1$.

If $O_1\neq O$, then the straight line $OO_1$ is invariant under the action of $\Gamma$, hence, $\Gamma$ acts reducibly on~$\mathbb{R}^n$.

Now, we suppose that $O_1=O$.
Then there is an unique symmetric positive definite operator
$A: \mathbb{R}^n \rightarrow \mathbb{R}^n$ such that $E=(Ax,x)=1$.

Since $E$ is $\Gamma$-invariant, all eigenspaces of $A$ are also $\Gamma$-invariant. Hence, if $A$ is not a multiple of the identity operator, then
$\Gamma$ acts reducibly on $\mathbb{R}^n$.

Otherwise, $A=a\cdot \operatorname{I}$ for some $a>0$, and $E$ should coincide with unit hypersphere in $\mathbb{R}^n$. Indeed, $E$ contains $P$ and has minimal volume
among hyperspheres with this property, but the vertices of $P$ are situated on the sphere $S(O,1)$.
\end{proof}
\smallskip

We immediately get the following result from Propositions \ref{pr.orap.1} and \ref{pr.orap.3}.

\begin{corollary}
Let $P$ be a homogeneous polytope and $E$ its L\"{o}wner~---~John ellipsoid in~$\mathbb{R}^n$. If $E$
does not coincide with the circumscribed sphere of $P$, then the full isometry group of $P$ is reducible and $P$ is not almost perfect.
\end{corollary}

\begin{example}\label{ex.parallelepiped}
A trivial example is a rectangular parallelotope in $\mathbb{R}^n,$ $n\geq 2,$ which is not a cube.
Indeed, if $E$ is the L\"{o}wner~---~John ellipsoid for a polytope $P$ and $\alpha: \mathbb{R}^n \rightarrow \mathbb{R}^n$ is an affine map, then
the ellipsoid $\alpha(E)$ is the L\"{o}wner~---~John ellipsoid for the polytope $\alpha(P)$.
\end{example}

\section{Examples of perfect polytopes}\label{sect.3}

The following result provides a powerful tool for proving the perfectness of some types of polytopes.

\begin{prop}\label{pr.partop}
Let $P$ be a convex hull of the points
$(\pm d_1, \pm d_2, \dots, \pm d_n)\in \mathbb{R}^n$, where $d_i>0$, $i=1,\dots,n$.
Suppose that there is a quadric $F(x)=0$ that contains all vertices of $P$.
Then there are some real numbers $A_i>0$, $i=1,\dots,n$, such that
$$
F(x)=A_1 x_1^2+A_2 x_2^2+A_3 x_3^2+\cdots +A_n x_n^2-\sum_{i=1}^n A_i{d\,_i^2}.
$$
\end{prop}

\begin{proof}
Recall that Example \ref{ex.partop} gives us a $(n-1)$-parameter family of ellipsoids containing the vertex set of $P$.
Now, we suppose that the quadric has the form (\ref{eq_cubepol1}). It is clear that $P$ is centrally symmetric with respect
to the origin $O.$ Then by Lemma \ref{le.sym}, $b_i=0$ for all indices. It remains to prove that $a_{ij}=0$ for all $i\neq j$.

Now, let fix some $i$ and let us prove that $a_{ij}=0$ for all $j\neq i$.
We consider the vertex $G=(\pm d_1, \pm d_2, \dots, \pm d_n)$, where the sign of the $i$-th coordinate is ``+", whereas the sign of the $j$-th coordinate for $j\neq i$
is ``--" if and only if $a_{ij}<0$.
We consider also the vertex $\widetilde{G}$, that differs from $G$ only in $i$-th coordinate
(it has $-d_i$ instead of $d_i$). We get
$$
0=F(G)-F(\widetilde{G})=4d_i \sum_{j\neq i} |a_{ij}|d_j.
$$
This implies that $a_{ij}=0$ if $i\neq j$.
Therefore, we can take $A_i=a_{ii}$ and the proposition is proved.
\end{proof}
\smallskip

\begin{corollary}\label{cor.partop}
Let $P$ be a homogeneous polytope in $\mathbb{R}^n$ with the following properties:

{\rm  1.} There is  a $n$-dimensional cube $C$, whose vertex set is the a proper subset of the vertex set of $P$, whereas the center of $C$ coincides with the center of $P$;

{\rm  2.} There is an edge $l$ of $C$, such that a supporting hyperplane $L$ to $P$, orthogonal to this edge, does not contain two vertices of $P$
with the distance between them equal to the length of $l$.

\begin{figure}[t]
\begin{minipage}[h]{0.48\textwidth}
\center{\includegraphics[width=0.9\textwidth, trim=0in 0in 0mm 0mm, clip]{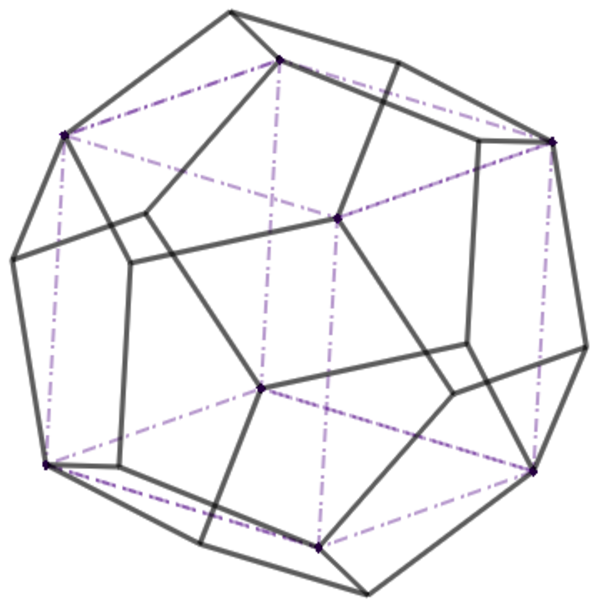} \\ a)}
\end{minipage}\quad
\begin{minipage}[h]{0.48\textwidth}
\center{\includegraphics[width=0.9\textwidth, trim=0in 0in 0mm 0mm, clip]{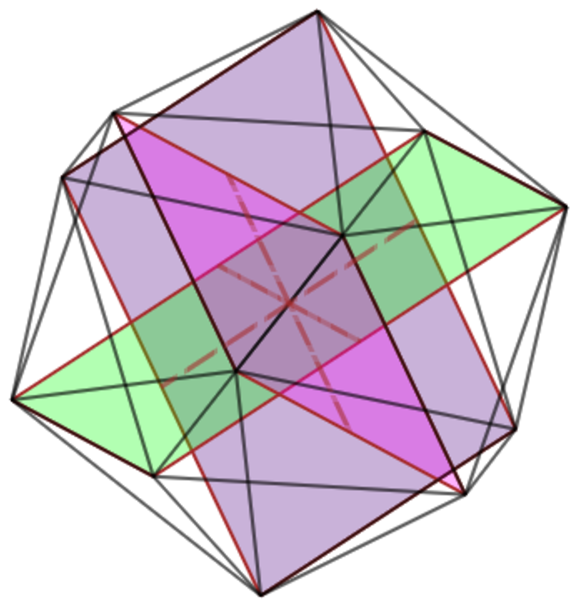} \\ b)}
\end{minipage}
\caption{Illustrations for:
a) Example \ref{ex:dodec1};
b) Proposition \ref{pr.icosa1}.
}
\label{Fig_dod_icos}
\end{figure}

Then $P$ is perfect.
\end{corollary}

\begin{proof}
Since the centers of $P$ and $C$ coincide, we may think that the origin $O$ is the center of both these polytope.
Without loss of generality, we may suppose that
$C=[-b,b]^n$ for some $b>0$. Let us consider any quadric $F(x)=0$ in $\mathbb{R}^n$, that contains all vertices of $P$. Therefore, it contains all vertices of $C$ and
$F(x)=\sum_{i=1}^n A_i x_i^2-\sum_{i=1}^n A_i\cdot b^2$ for some $A_i \in \mathbb{R}$, $i=1,\dots,n$, by Proposition \ref{pr.partop}.

Let us consider a positive symmetric operator $A :\mathbb{R}^n \rightarrow\mathbb{R}^n$ such that $F(x)=(Ax,x)-\operatorname{trace}(A) b^2$.
We see that the operator $A$  has a diagonal form in the basis of vectors that are parallel to the edges of the cube~$C$ (the standard basis of $\mathbb{R}^n$).
The same property is fulfilled with respect to any cube of the form $a(C)$, where $a$ is an isometry of $P$ by  Proposition \ref{pr.partop} again.

Without loss of generality, we may suppose that the edge $l$ is parallel to the vector $(0,0,\dots,0,1)\in\mathbb{R}^n$ and the hyperplane $L$ has the equation
$x_n=\mu$, where $\mu < -b$.

Now, we consider a vertex $V$ of $P$ from the hyperplane $L$. Since $P$ is homogeneous, then there is an isometry $a$ of $P$ such that
$a(V_0)=V$, where $V_0=(-b,-b,\dots,-b)$. Obviously, $a(C)\neq C$. Let us consider the following vertices of $C$ (and $P$): $V_i$, $i=1,\dots n$,
such that all coordinates of $V_i$ is $-b$ but the $i$-th one is $b$.

By the assumptions of the corollary, any vector of the type $\overrightarrow{V a(V_i)}=\overrightarrow{a(V_0) a(V_i)}$ has positive $n$-th coordinate.
It is clear that the vector $\overrightarrow{V_0 V_n}$ is an eigenvector of the operator $A$ with the eigenvalue $A_n$ (see the expression for $F(x)$).
On the other hand, all vectors $\overrightarrow{V a(V_i)}=\overrightarrow{a(V_0) a(V_i)}$, $i=1,\dots,n$, are also the eigenvectors for the operator $A$
by Proposition \ref{pr.partop}. Each of these vectors has a positive scalar product with the vector $\overrightarrow{V_0 V_n}$, which implies that
all these vectors have one and the same eigenvalue $A_n$ (as the vector $\overrightarrow{V_0 V_n}$ has). This means that
$F(x)=A_n\Bigl(\sum_{i=1}^n x_i^2-n\cdot b^2\Bigr)$, the quadric $F(x)=0$ coincides with the circumscribed sphere
of $P$. Therefore, $P$ is perfect.
\end{proof}

\begin{example}\label{ex:dodec1}
The dodecahedron $D$ is a perfect polytope. Indeed,
every diagonal of any pentagonal face of the dodecahedron is an edge of some cube with vertices from the vertex set of this dodecahedron,
see e.~g. \cite[P.~76]{FToth64}. Any supporting plane for~$D$, that is orthogonal to an edge of this cube,
contains exactly two adjacent vertices of~$D$, see Fig.~\ref{Fig_dod_icos}~a). Therefore, the dodecahedron is perfect by Corollary \ref{cor.partop}.
\end{example}

\begin{prop}\label{pr.icosa1}
The icosahedron is a perfect polytope.
\end{prop}

\begin{proof}
If we take a section of the icosahedron $I$ through its center and any its edge, then we get a centrally symmetric $6$-gon, whose four vertices determines
a rectangle with the side proportion $1:\varphi$,
where $\varphi=(1+\sqrt{5})/2\approx 1,618033988$ is the so-called golden ratio. Moreover, a given edge is one of short sides of this rectangle
(the other short side is also an edge of $I$) and there are
two other edges of $I$ that are orthogonal to it and each to other, see Fig.~\ref{Fig_dod_icos}~b).

Therefore, we may assume without loss of generality that the coordinates is such that the origin $O$ coincides with the center of $I$
and the axes are parallel to some edges of $I$.

Let us suppose that a quadric $F(x)=0$ contains every vertex of $I$.
It is clear that
$$
F(x)=\sum_{i,j=1}^3 a_{ij} x_ix_j+\sum_{i=1}^3 b_i x_i+ c
$$
for some  $a_{ij},b_i \in \mathbb{R}$,  $a_{ij}=a_{ji}$,  $i,j=1,2,3$.

The above discussion imply that for $x_3=0$ we get a rectangle with vertex from the vertex set of $I$.
Note that $F(x_1,x_2,0)=a_{11}x_1^2+2a_{12}x_1x_2+a_{22}x_2^2+b_1x_1+b_2x_2+c$. Now, Proposition \ref{pr.partop} implies that
$a_{12}=b_1=b_2=0$. The same arguments can be repeated for $x_2=0$ and $x_1=0$. From these plane sections we get $a_{13}=b_1=b_3=0$ and  $a_{23}=b_2=b_3=0$ respectively.
Hence, we have
\begin{equation}\label{eq.isos_1}
F(x_1,x_2,x_3)=a_{11}x_1^2+a_{22}x_2^2+a_{33}x_3^2+c.
\end{equation}
The same result we get in any coordinate system with the origin $O$ and axes parallel to some three edges of $I$.

Now, recall that all edges of $I$ with a common vertex are not orthogonal each to other. Let us  consider all edges $v_1, v_2, v_3, v_4, v_5$
of $I$ with a common vertex $V$ and apply the above argument to every coordinate system, where the first axis is parallel to $v_i$, $i=1,\dots,5$.
For all these coordinate systems, we obtain \eqref{eq.isos_1}. This means that the quadratic form $\sum_{i,j=1}^3 a_{ij} x_ix_j$ has a diagonal
form in every such coordinate system. But this is possible only when $a_{11}=a_{22}=a_{33}$. Therefore,
$F(x_1,x_2,x_3)=a_{11}(x_1^2+x_2^2+x_3^2)+c$, and the quadric $F(x_1,x_2,x_3)=0$ is the circumscribed sphere. Consequently, the icosahedron $I$ is perfect.
\end{proof}
\smallskip

The following theorem gives the classification of regular polyhedra in $\mathbb{R}^3$ with respect to the properties to be perfect or almost perfect.

\begin{theorem}\label{th.regpolyh1}
The tetrahedron, the cube, and the octahedron are not almost perfect, while the dodecahedron and the icosahedron are perfect.
\end{theorem}

\begin{proof}
The proof follows from Proposition \ref{pr.regsimpl1}, Example \ref{ex.partop}, Example \ref{ex.orthopl1}, Example \ref{ex:dodec1}, and Proposition \ref{pr.icosa1}.
\end{proof}
\smallskip

Let us consider one more simple but useful result.

\begin{prop}\label{prop.partopn}
Let us suppose that homogenous polytopes $P$ and $P_1$ in $\mathbb{R}^n$ are non-degenerate, have one and the same center $O$, and
any vertex of $P_1$ is a vertex of $P$. Then, if $P_1$ is {\rm(}almost{\rm)} perfect, then $P$ is also {\rm(}almost{\rm)} perfect.
\end{prop}

\begin{proof}
If all vertices of $P$ are on some (symmetric with respect to $O$) quadric $E$, then all vertices of $P_1$ are on $E$.
Since $P_1$ is (almost) perfect, then $E$ is the circumscribed hypersphere for $P_1$, which is also the circumscribed hypersphere for $P$.
Therefore, $P$ is (almost) perfect.
\end{proof}

\medskip

Now, we a going to discuss $4$-dimensional regular polytopes. All necessary details on these polytopes can be found in Section~3 of \cite{BerNik21}.
We can describe these polytopes in term of quaternions (we identify the noncommutative division ring of quaternions $\mathbb{H}$ with~$\mathbb{R}^4$ via the map
$x=x_1+x_2 {\bf i}+x_3 {\bf j} + x_4 {\bf k} \mapsto (x_1,x_2,x_3,x_4)$).

\begin{theorem}\label{th.fourdim1}
Let $P$ be a $4$-dimensional regular polytope. Then the following assertions hold:

{\rm  1.} $P$ is not almost perfect if $P$ is the $4$-dimensional simplex, cube and orthoplex;

{\rm  2.} $P$ is perfect if $P$ is the $24$-cell, $120$-cell or $600$-cell.
\end{theorem}

\begin{proof}
We already know that the $4$-dimensional simplex, cube and orthoplex are not almost perfect.
Therefore, it suffices to consider only the $24$-cell, $120$-cell, and $600$-cell.
It suffices to check conditions of Corollary \ref{cor.partop}, where $P$ is any of these polytopes.

The $24$-cell can be constructed as convex hull of the following
finite subset of unit quaternions (that constitutes {\it the binary tetrahedral group} $BT$, a subgroup in the group of unit quaternions):
\begin{equation*}
BT= \left\{\pm {\bf 1},\pm {\bf i},\pm {\bf j},\pm  {\bf k},
{\frac{1}{2}}\bigl(\pm  {\bf 1}\pm  {\bf i}\pm  {\bf j}\pm  {\bf k}\bigr)
\right\},
\end{equation*}
It is clear that this set contains exactly $24$ points,
while the set of quaternions \linebreak
$\left\{{\frac{1}{2}}\bigl(\pm  {\bf 1}\pm  {\bf i}\pm
{\bf j}\pm  {\bf k}\bigr)\right\}$ forms a $4$-dimensional cube (hypercube, tesseract) $C$.

Moreover, let us consider the edge $l$ of $C$ that is parallel to the quaternion ${\bf 1}$ and the supporting hyperplane $\operatorname{Re}(u)=-1$
(the real part of the quaternion $u$ is $-1$).
We see that a unique vertex of the $24$-cell in this hyperplane is $-{\bf 1}$.
Hence, we get that the $24$-cell is perfect by Corollary \ref{cor.partop}.

Let us recall that {\it the binary icosahedral group} $BI$ is the union
of the group~$BT$ with the set of unit quaternions obtained
by even permutations of coordinates (all possible combinations)
from the quaternions
\begin{equation*}
\bigl(0\pm {\bf i}  \pm \varphi^{-1}{\bf j} \pm  \varphi\, {\bf k}\bigr),
\end{equation*}
where $\varphi=(1+\sqrt{5})/2$ is the golden ratio.
There are 96 such quaternions in total, so $|BI|=120$.
The convex hull of elements of~$BI$ (that have 120 elements) in~$\mathbb{R}^4$
forms the 600-cell (hexacosichor).
Since $BT \subset BI$ and the $24$-cell is perfect, we get that
$600$-cell is also perfect by Proposition \ref{prop.partopn}.

Let us consider a representation of the vertex set of the $120$-cell in the coordinates as in \cite[P.~156--157]{Cox73}.
Assuming that the origin $O$ is the center of the $120$-cell
and the radius of the circumscribed sphere is $\sqrt{8}$, we get the following points:
\begin{eqnarray*}
(\{0, 0, \pm 2, \pm 2\}), \quad ([0, \pm {\varphi}^{-1}, \pm {\varphi}, \pm \sqrt{5}]),  \quad (\{\pm {\varphi}, \pm {\varphi}, \pm {\varphi}, \pm {\varphi}^{-2}\}), \\
(\{\pm 1, \pm 1, \pm 1, \pm \sqrt{5}\}),
 \quad
(\{\pm {\varphi}^{-1},  \pm {\varphi}^{-1}, \pm {\varphi}^{-1}, \pm {\varphi}^{2}\}),  \\
([0, \pm {\varphi}^{-2}, \pm 1, \pm {\varphi}^{2}]),  \quad ([\pm {\varphi}^{-1}, \pm 1, \pm {\varphi}, \pm 2]),
\end{eqnarray*}
where $\varphi=(1+\sqrt{5})/2$ is the golden ratio, the symbols  $\{\,\cdots\,\}$ and $[\,\cdots \,]$ are used respectively for all permutations and even
permutations of the elements inside.
The above sets contain $24$, $96$, $120$, $120$, $120$, $96$, $24$ points respectively (together exactly $600$ vertices of the $120$-cell).

It is easy to see also that the vertices $(\{0, 0, \pm 2, \pm 2\})$ determine a $24$-cell, see \cite[P.~156]{Cox73}.
Hence, the $120$-cell is perfect by Proposition \ref{prop.partopn}. The theorem is proved.
\end{proof}

\smallskip

\begin{remark}
Similar arguments show that {\it the dysphenoidal $288$-cell} is perfect.
In fact, it can be represented as the convex hull of
{\it the binary octahedral group\/}
$BO$, that contains $BT$ as a subgroup.
It should be noted that the dysphenoidal $288$-cell is a homogeneous polytope, but is not regular or semiregular,
see details in Section~3 of \cite{BerNik21}.
\end{remark}

\begin{remark}
It is known also that some vertices of the $600$-cell constitute the vertex set of the $120$-cell, see Theorem 1 in \cite{AZL2003}.
\end{remark}

Theorem 1 in \cite{AZL2003} is interesting in some other respects. We shall state it in a little different words.

\begin{definition}
A $(n\times n)$-matrix $A$, $n\geq 2,$ with elements $\pm 1$ such that $AA^T=nI_n$, where $I_n$ is the unit $(n\times n)$-matrix, is called
{\it Hadamard matrix} (of order $n$). This is equivalent to the orthogonality of different columns or rows of the matrix $A$.
\end{definition}

The Hadamard matrices must have an order $n=2$ or $n=4k$. The famous ``Hadamard conjecture'' states that there exists
an Hadamard matrix of every order $n=4k.$ Now the least order $n$ for which is unknown the existence of Hadamard matrix, is $668=4\cdot 167$.

\begin{theorem} [\cite{AZL2003}]
For a natural number $n\geq 2$ the following conditions are equivalent:

{\rm  1.} There exists a Hadamard matrix of order $n$,

{\rm  2.} The vertex set of any regular $(n-1)$-dimensional simplex is a subset of the vertex set of some $(n-1)$-dimensional cube,

{\rm  3.} The vertex set of any regular $n$-dimensional orthoplex is a subset of some $n$-dimensional cube.

Proper inclusions of vertex sets of other regular polytopes exist only for dimensions $k=2,3,4$. For $k=3$, the vertex set of any cube is a subset of the vertex set
for a regular dodecahedron. For $k=4$, a regular polytope may have altogether $5$, $8$, $16$, $24$, $120$ or $600$ vertices. The vertex set of the simplex is a subset
of the vertex set only for $120$-cell {\rm(}which has $600$ vertices{\rm)}. The vertex set of any other polytope is a subset of the vertex set of arbitrary polytope
with larger quantity of vertices.
\end{theorem}

\begin{prop}\label{pr.prism}
There is a right $n$-dimensional homogeneous prism $P$ in~$\mathbb{R}^n$, $n=3,4,5$,
whose full isometry group $\Gamma_P$ is reducible and the L\"{o}wner~---~John ellipsoid coincides with circumscribed sphere of $P$.
Up to similarity, there are at least four such prisms $P$ for $n=5$. By Proposition \ref{pr.notap}, all such $P$ are not almost perfect.
\end{prop}

\begin{proof}
Let $M$ be any of the following $(n-1)$-dimensional polytopes: the $2$-dimensional regular octagon, $3$-dimensional regular dodecadron, $4$-dimensional regular
$24$-cell, $120$-cell or $600$-cell, or homogeneous dysphenoidal $288$-cell. The vertex set of every such polytope $M$ contains as a subset the vertex set of a cube
$C_0$ with the same dimension such that the center of $C_0$ coincides with the center of $M$. Now take as $P$ the right prism in $\mathbb{R}^n$ with the base
$M$ and the height $a$, which is equal to the edge length of $C_0$.
Then the vertex set of $P$ contains the vertex set of a cube $C$ with edge lengths $a$ such that the center of
$C$ coincides with the center of $P$. It is clear that $P$ is homogeneous and the full isometry group of $P$ is reducible. At the same time,
the L\"{o}wner~---~John ellipsoid of $P$ coincides with circumscribed sphere $S$ of $P$. This follows from the facts that $S$ is also
the circumscribed sphere and the L\"{o}wner~---~John ellipsoid for $C$, while every ellipsoid which contains $P$ also contains $C$. The last two statements are clear.
\end{proof}

\begin{remark}
The prisms $P$ from the last proposition show that the first condition in Corollary \ref{cor.partop}  is not enough for
its concluding statement. In particular, the first condition does not imply the second one.
\end{remark}

There are reasons to believe that the following conjecture is correct.

\begin{conjecture}\label{con.partop}
Let $P$ be a homogeneous polytope in $\mathbb{R}^n$ with the following properties:

{\rm  1.} There is  a $n$-dimensional cube $C$ whose vertex set is a proper subset of the vertex set of $P$;

{\rm  2.} There is no subgroup of the symmetry group of $P$ that is transitive on the vertex set of $P$ and acts reducibly on $\mathbb{R}^n$.

Then $P$ is perfect.
\end{conjecture}

\medskip

Let us discuss now semiregular non-regular polytopes in $\mathbb{R}^n$ for $n\geq 4$. They exist only for $4\leq n \leq 8$ and are called
the {\it Gosset polytopes} \cite{Gos}, \cite{Elte12}, \cite{BerNik21n}. Up to similarity, there are three four-dimensional Gosset polytopes and
a unique Gosset polytope in $\mathbb{R}^n$ for $n\in \{5,6,7,8 \}$, which we denote by  $\operatorname {Goss\,}_n$.
Detailed descriptions of Gosset polytopes could be found  in Sections 5 --- 9 of \cite{BerNik21n}.

The Gosset polytope $\operatorname {Goss\,}_5$ is the demihypercube, it is not almost perfect by Example~\ref{ex.hpartop}.
Gosset polytopes $\operatorname {Goss\,}_n,$ $n\geq 5$, has respectively $16$, $27$, $56$, and $240$ vertices for $n=5,6,7,8$ which are less
than the numbers of vertices of the $n$-dimensional cube, equal to $2^n$. So, it is impossible to apply Corollary \ref{cor.partop} to these
Gosset polytopes.

The Gosset polytopes in $\mathbb{R}^4$ are full truncations, that is, the convex hulls of the midpoints of the edges of the four-dimensional
regular simplex and $600$-cell, as well as the snub $24$-cell. They have respectively $10$, $720$, and $96$ vertices altogether.

\begin{question}
Is it true that the vertex set of $4$-dimensional cube is a subset of the vertex set of the snub $24$-cell or the truncated $600$-cell?
\end{question}

It is very interesting that the vertex sets of all but two Gosset polytopes could be described uniformly as some subsets of the lattice $E_8$ in $\mathbb{R}^8.$
According to Chap. 14, sect. 1 of \cite{ConSlo88}, in the lattice packing in $\mathbb{R}^8$ with centers at the points of the lattice $E_8$ there are
240 balls tangent to a single ball, 56 balls tangent to either of balls in a pairs of tangent balls, 27 balls tangent to every ball in a triple of
pairwise tangent balls, 16 balls tangent to to every ball of a quadruple of pairwise tangent balls, 10 balls tangent to every ball from a quintuple of pairwise
tangent balls, and 6 balls tangent to every ball from a hextuple of pairwise touching balls. The centers of mentioned 240, 56, 27, 16, 10, and 6 balls
are respectively the vertex sets of $\operatorname {Goss\,}_8$, $\operatorname {Goss\,}_7$, $\operatorname {Goss\,}_6$, $\operatorname {Goss\,}_5$, the full truncation of 4-dimensional regular simplex, and the octahedron in $\mathbb{R}^3.$

The last statement implies also that the vertex figure of $\operatorname {Goss\,}_8$ is $\operatorname {Goss\,}_7$, the vertex figure of $\operatorname {Goss\,}_7$
is $\operatorname {Goss\,}_6$, the vertex figure of $\operatorname {Goss\,}_6$ is $\operatorname {Goss\,}_5$, the vertex figure of $\operatorname {Goss\,}_5$ is
the full truncation of the regular $4$-dimensional simplex, and the vertex figure of the last polytope is the octahedron.
In particular, the vertex set of
the full truncation of the regular $4$-dimensional simplex can be considered as the set of all vertices in the standard
$5$-dimensional cube with exactly two coordinates equal to $-1$.

Now, we are going to consider a $6$-dimensional example of perfect polytope, that is $\operatorname {Goss\,}_6$.
This result could be proved by simple methods of linear algebra (see a discussion in \cite{DER2007,Dut2017}).
A simple proof of the fact that $\operatorname {Goss\,}_6$ is almost perfect can be found in \cite[Theorem 6]{BerNik21n}.
In the next example we consider some argument to prove the perfectness of $\operatorname {Goss\,}_6$ without computer calculations.

\begin{example}\label{ex:Gos6.1}
Let us consider a brief explicit description of  $\operatorname {Goss\,}_6$.
This polytope can be implemented in different ways. Let us set it with the coordinates of the vertices in $\mathbb{R}^6$, as it is done in~\cite{Elte12}.
Let us put $a=\frac{\sqrt{2}}{4}$ and $b=\frac{\sqrt{6}}{12}$. We define the points $A_i \in \mathbb{R}^6$, $i=1,\dots,27$, as follows:

\smallskip
\noindent
{\small
\begin{tabular}{lll}
$A_1=(0,0,0,0,0,4b)$, & $A_2=(a,a,a,a,a,b)$, & $A_3=(-a,-a,a,a,a,b)$,\\$A_4=(-a,a,-a,a,a,b)$, & $A_5=(-a,a,a,-a,a,b)$, & $A_6=(-a,a,a,a,-a,b)$,\\
$A_7=(a,-a,-a,a,a,b)$, & $A_8=(a,-a,a,-a,a,b)$, & $A_9=(a,-a,a,a,-a,b)$,\\
$A_{10}=(a,a,-a,-a,a,b)$, & $A_{11}=(a,a,-a,a,-a,b)$, & $A_{12}=(a,a,a,-a,-a,b)$,\\
$A_{13}=(-a,-a,-a,-a,a,b)$, & $A_{14}=(-a,-a,-a,a,-a,b)$, & $A_{15}=(-a,-a,a,-a,-a,b)$,\\
$A_{16}=(-a,a,-a,-a,-a,b)$, & $A_{17}=(a,-a,-a,-a,-a,b)$, & $A_{18}=(2a,0,0,0,0,-2b)$,\\
$A_{19}=(0,2a,0,0,0,-2b)$, & $A_{20}=(0,0,2a,0,0,-2b)$, & $A_{21}=(0,0,0,2a,0,-2b)$,\\
$A_{22}=(0,0,0,0,2a,-2b)$, & $A_{23}=(-2a,0,0,0,0,-2b)$, & $A_{24}=(0,-2a,0,0,0,-2b)$,\\
$A_{25}=(0,0,-2a,0,0,-2b)$, & $A_{26}=(0,0,0,-2a,0,-2b)$, & $A_{27}=(0,0,0,0,-2a,-2b).$\\
\end{tabular}}
\medskip

The Gosset polytope $\operatorname{Goss\,}_6$ is the convex hull of these points.
It is easy to check that $d(A_1, A_i)=1$ for $2\leq i \leq 17$ and $d(A_1, A_i)=\sqrt{2}$ for $18\leq  i \leq 27$.

It is clear that the points $A_2 - A_{17}$ are vertices of a five-dimensional demihypercube (the corresponding hypercube has $32$ vertices of the form
$(\pm a, \pm a,\pm a,\pm a,\pm a,b)$),
and the points $A_{18} - A_{27}$ are the vertices of the five-dimensional hyperoctahedron (orthoplex),
which is a facet of the polytope $\operatorname{Goss\,}_6$ (lying in the hyperplane $x_6=-2b$).
The origin $O=(0,0,0,0,0,0)\in \mathbb{R}^6$ is the center of the hypersphere described around $\operatorname{Goss\,}_6$ with radius $4b=\sqrt{2/3}$.

Let us suppose that there is a quadric $Q$ in $\mathbb{R}^6$, containing all vertices of $\operatorname{Goss\,}_6$.
We may suppose that $Q$ is determined by the equation $F(x_1,x_2,x_3,x_4,x_5,x_6)=0$, where
\begin{equation}\label{eq_gos6pol1}
F(x)=\sum_{i,j=1}^6 a_{ij} x_ix_j+\sum_{i=1}^6 b_i x_i+ c,
\end{equation}
$a_{ij},b_i,c \in \mathbb{R}$, and $a_{ij}=a_{ji}$ for all indices.
It is easy to check that
$$
0=F(A_2)+F(A_5)+F(A_7)+F(A_{13})-F(A_3)-F(A_4)-F(A_8)-F(A_{10})=16 a^2 \bigl(a_{14}+a_{23}\bigr).
$$
Since we can freely use all substitutions of first five coordinates, similar argument shows that $a_{ij}+a_{kl}=0$, where all $1\leq i,j,k,l \leq 5$ are pairwise distinct.
In particular, if $\{i,j,k,l,m\}=\{1,2,3,4,5\}$, then $a_{km}=-a_{ij}=a_{kl}$. This means that all non-diagonal elements in every row or any column
(recall that $a_{ij}=a_{ji}$) of the matrix
$\bigl(a_{ij}\bigr)_{ij=1}^5$ are identical. Then $a_{ij}+a_{kl}=0$ implies that all non-diagonal elements of this matrix are zero.

Note also that $0=F(A_{18})-F(A_{23})=4 a \cdot \bigl( b_1-4 a_{16} b\bigr)$ and
$0=F(A_{18})+F(A_{23})=2 a \cdot \bigl( c- 2b \cdot b_6 +4b^2 \cdot a_{66}+ 4a^2\cdot  a_{11}\bigr)$.
Since we can use all substitutions of first five coordinates, similar argument shows that
$b_i=4 b \cdot a_{i6}$ and $4 a^2\cdot a_{ii}=2b\cdot b_6- c - 4 b^2\cdot a_{66}$ for all $i=1,\dots,5$. In particular, $a_{11}=a_{22}=a_{33}=a_{44}=a_{55}$.

Since $0=F(A_2)+F(A_7)+F(A_3)+F(A_4)=4 a \cdot \bigl(2(a_{15} + a_{14} + a_{23})a + b_1 + 2b\cdot a_{16}\bigr)=4 a \cdot \bigl(b_1 + 2b\cdot a_{16}\bigr)$,
then we get $b_1 =- 2b\cdot a_{16}$ and (using substitutions of first five coordinates) $b_i =- 2b\cdot a_{i6}$ for $i=1,\dots,5$.
From this and $b_i=4 b \cdot a_{i6}$ we easily get that  $b_i=a_{i6}=0$ for all $1\leq i \leq 5$.

It should be noted also that $0=F(A_1)= c+4 b\cdot b_6+16b^2\cdot a_{66}$ and
$0=\sum_{i=2}^{17} F(A_i)-F(A_1)=12b \cdot b_6+2 \sum_{i=1}^5 a_{ii}+15c=0$.
All the above arguments imply that $b_6=0$, $c=-1$, and
$a_{ii}=\frac{3}{2}$, $1\leq i \leq 6$.

Finally, we get
$F(x)=\frac{3}{2} \bigl(x_1^2+x_2^2+x_3^2+x_4^2+x_5^2+x_6^2\bigr)-1$. Therefore, the Gosset polytope $\operatorname{Goss\,}_6$ is perfect.
\end{example}

\begin{remark} If we know that  $b_i=0$ for all $i=1,\dots,5$, we can modify the proof of $a_{ij}=0$ for $i\neq j$. Indeed, since
$F(x)=\sum_{i,j=1}^6 a_{ij} x_ix_j+ c$, we see that $F(A)=0$ for all points of the type $A=(\pm a,\pm a, \pm a,\pm a,\pm a,b)$,
that constitute a $5$-dimensional cube in the hyperplane $x_6=b$. Using Proposition \ref{pr.partop}, we easily prove that $a_{ij}=0$ for all $i,j=1,\dots,5$, $i\neq j$.
\end{remark}

\begin{prop}\label{pr.dist}
Let $P$ be a non-degenerate homogeneous polytope in $\mathbb{R}^n,$ $n\geq 3,$ with the vertex set $V$ and a vertex $v\in V$ has at least two different
distance spheres in $V$, whose convex hulls are $(n-1)$-dimensional perfect polytopes. Then $P$ is perfect.
\end{prop}

\begin{proof}
Up to a similarity, we may assume that the circumscribed sphere $S$ of $P$ is the unit sphere in $\mathbb{R}^n$ with the center at the origin $O$
and $v=(0,\dots,0,1)$. If $\sigma_1$, $\sigma_2$ are the mentioned above distance spheres for $v_1$ in $V$ and $P_1$, $P_2$ are the convex hulls of $\sigma_1$, $\sigma_2$,
then the circumscribed spheres $S_1$, $S_2$ of $P_1$, $P_2$ lie respectively in some hyperplanes $x_1=a_1$, $x_1=a_2$ in $\mathbb{R}^n$, where $-1< a_1< a_2<1$.

Now let $E$ be any second order hypersurface in $\mathbb{R}^n$, which contains all vertices of~$P$.
Since $P_1$ and $P_2$ are perfect, then the intersections of $E$ with the hyperplanes
$x_1=a_1$ and $x_1=a_2$ in $\mathbb{R}^n$ are the circumscribed spheres $S_1$, $S_2$ for  $P_1$, $P_2$ respectively.

The above considerations imply that the intersection of $E$ with any two-dimensional plane $\Sigma^2$ in $\mathbb{R}^n$, which includes the $x_1$-axis,
is a second order curve $c$ in $\Sigma^2$ which includes the following quintuple of points:
$Q_5=\{v_5\}\bigcup \bigl(S_1\cap \Sigma^2\bigr)\bigcup \bigl(S_2\cap \Sigma^2\bigr)$. Since the unit circle $S^1$ in $\Sigma^2$ with the center $O$ also contains $Q_5$,
then $c=S^1$ by Lemma \ref{le.four}. Therefore, $E$
includes the hypersurface $S$ in $\mathbb{R}^n$, obtained by rotations of such a circle $S^1$ around the $x_1$-axis, so $E=S$.
On the other hand, $S$ is the circumscribed sphere of~$P$. Therefore, $P$ is perfect.
\end{proof}

\begin{theorem}\label{th.goss}
The Gosset polytopes $\operatorname{Goss\,}_7$ and $\operatorname{Goss\,}_8$, the regular icosahedron, and the regular dodecahedron are perfect.
The fully truncated regular $4$-dimensional simplex and $\operatorname{Goss\,}_5$ are not almost perfect.
\end{theorem}

\begin{proof}
A detailed information on $\operatorname {Goss\,}_7$ and $\operatorname {Goss\,}_8$ could be found in \cite{BerNik21n}.
We recall some information which is quite sufficient for subsequent considerations.

The polytopes $P_7=\operatorname {Goss\,}_7$ and $P_8=\operatorname {Goss\,}_8,$ and the  icosahedron $I$ are symmetric with respect to their centers and
their vertex figures for some opposite vertices $v$ and $-v$ are two polytopes which are the convex hulls of two different distance spheres
in the corresponding vertex set for $v$ and congruent respectively to $P_6=\operatorname {Goss\,}_6$, $P_7$, and the regular pentagon $Q_5$.
Since $P_6$ and $Q_5$ are perfect by Example \ref{ex:Gos6.1} and Proposition~\ref{pr.regpolyg1} respectively, then $P_7$ and $I$ are perfect by Proposition~\ref{pr.dist}.
Applying once more Proposition~\ref{pr.dist} to $P_8$,
we get that $\operatorname {Goss\,}_8$ is also a perfect polytope.

The dodecahedron $D$ is symmetric with respect to its center $O$. Let $v$ and $-v$ are two its opposite vertices in the vertex set $V$ of $D$ and $\sigma$
be the set of all ends of diagonals with the origin $v$ in three faces (regular pentagons) of $D$ with the vertex $v$. Clearly, $\sigma$ and $-\sigma$ are two different
distance spheres for $v$ in $V$ which are not regular hexagons but some orbits with respect to the isotropy subgroup of $v$ in the full isometry group of $D$.
All vertices of these hexagons lie in their circumscribed sphere, so they are perfect polytopes by Lemma~\ref{le.four}. Now $D$ is a perfect polytope by
Proposition~\ref{pr.dist}.

We described above the vertex sets for $\operatorname {Goss\,}_5$ and for the full truncation of the regular $4$-dimensional simplex as some subsets of the vertex set
for the standard cube in $\mathbb{R}^5$. It is easy to see from these descriptions that there are some subgroups of commutative group $\mathbb{Z}_2^{\,\,5}$,
which act transitively on the vertex sets of these two polytopes. Then they are not almost perfect by Proposition~\ref{pr.comm}.
\end{proof}


\vspace{4mm}

\begin{thebibliography}{99}
\bibitem{AZL2003}
{\sl Adams~J., Zvengrowski~P., Laird~P.}
Vertex embeddings of regular polytopes~//
{\it Expo. Math.}, 2003, V.\,21, N\,4, P.~339--353, {\bf MR}2022003, {\bf Zbl.}1042.52010.



\bibitem{Ball97}
{\sl Ball K.}
{\it An elementary introduction to modern convex geometry}, In:
Levy, Silvio (ed.), Flavors of geometry. Cambridge: Cambridge University Press. Math. Sci. Res. Inst. Publ. 31, 1997, P.~1--58, {\bf MR}1491097, {\bf Zbl.}0901.52002.



\bibitem{BarBl05}
{\sl Barvinok~A., Blekherman~G.}
{\it Convex geometry of orbits.} In: Goodman, Jacob Eli (ed.) et al., Combinatorial and computational geometry, 51--77, Math. Sci. Res. Inst. Publ., 52,
Cambridge Univ. Press, Cambridge, 2005,  {\bf MR}2178312, {\bf Zbl.}1096.52002.


\bibitem{BerNik19}
{\sl Berestovskii V.~N., Nikonorov Yu.~G.}
Finite homogeneous metric spaces~//
{\it Sib. Mat. Zh.}, 2019, V.\,60, N\,5, P.\,973--995 (in Russian). English translation:
Siberian Math. J., 2019, V.\,60, N\,5, P.\,757--773, {\bf MR}4055423, {\bf Zbl.}1431.51008.
DOI: 10.1134/S0037446619050021


\bibitem{BerNik20}
{\sl Berestovskii V.~N., Nikonorov Yu.~G.}
{\it Riemannian Manifolds and Homogeneous Geodesics},
Springer Monographs in Mathematics. Cham: Springer, 2020,  {\bf MR}4179589,   {\bf Zbl.}1460.53001.
DOI: 10.1007/978-3-030-56658-6

\bibitem{BerNik21}
{\sl Berestovskii V.~N., Nikonorov Yu.~G.}
Finite homogeneous subspaces of Euclidean spaces~//
{\it Mat. Trudy}, 2021, V.\,24, N\,1. P.\,3–-34  (in Russian).~//
English translation:
Siberian Advances in Mathematics, 2021, V.\,31, N\,3, P.\,155--176, {\bf Zbl.}07656933.
DOI: 10.1134/S1055134421030019

\bibitem{BerNik21n}
{\sl Berestovskii V.~N., Nikonorov Yu.~G.}
Semiregular Gosset polytopes~// {\it Izv. Ross. Akad. Nauk, Ser. Mat.}, 2022, V.\,86, N\,4, P.\,51--84 (in Russian).
English translation:  Izv. Math., 2022, V.\,86, N\,4, P.\,667--698.
DOI: 10.1070/IM9169


\bibitem{BerNik21nn}
{\sl Berestovskii V.~N., Nikonorov Yu.~G.}
On finite homogeneous metric spaces~// {\it Vladikavkaz Mathematical Journal}, 2022, V.\,2, N\,2, P.\,51--61, {\bf MR}4448043, {\bf Zbl.}07598648.
DOI: 10.46698/h7670-4977-9928-z

\bibitem{BerNik22}
{\sl Berestovskii V.~N., Nikonorov Yu.~G.}
On $m$-point homogeneous polytopes in Euclidean spaces~// {\it Filomat}, 2023,  V.\,37, N\,25 (accepted), see also arXiv:2206.13096.



\bibitem{BoFe1987}
{\sl Bonnesen~T., Fenchel~W.}
{\it Theory of Convex Bodies}, BCS Associates, Moscow, ID, 1987. Translated
from the German and edited by L. Boron, C. Christenson and B. Smith., {\bf MR}0920366, {\bf Zbl.}0628.52001.

\bibitem{ConSlo88}
{\sl Conway~J.H., Sloane~N.J.A.} {\it Sphere packings, lattices, and groups}, Grundlehren Math. Wiss., vol. 290, Springer-Verlag, New York, 1988
, {\bf MR}0920369, {\bf Zbl.}0634.52002;
Russian transl., vol. 1,2, Mir, Moscow, 1990.


\bibitem{Cox73}
{\sl Coxeter~H.S.M.} {\it Regular polytopes, 3d ed.} New York: Dover, 1973, {\bf MR}0370327.


\bibitem{Dut2017}
{\sl Dutour~Sikiri\'{c}~M.}
The seven dimensional perfect Delaunay polytopes and Delaunay simplices,
Canadian J. Math., 2017, V.\,69, N\,5, P.\,1143--1168, {\bf MR}3693151, {\bf Zbl.}1381.11054.

\bibitem{DER2007}
{\sl Dutour~Sikiri\'{c}~M., Erdahl~R., Rybnikov~K.},
Perfect Delaunay polytopes in low dimensions, {\it  Integers}, 2007, 7, A39, 49 pp., {\bf MR}2342197  {\bf Zbl.}1194.52018.



\bibitem{Elte12}
{\sl Elte~E.~L.}
{\it The Semiregular Polytopes of the Hyperspaces},
1912, Groningen: University of Groningen,
\url{https://quod.lib.umich.edu/u/umhistmath/ABR2632.0001.001}, 29.03.2023.


\bibitem{FToth64}
{\sl Fejes T{\'o}th L.}
{\it Regular figures.}
International Series of Monographs on Pure and Applied Mathematics. 47. Oxford etc.: Pergamon Press, 1964, {\bf MR}0165423, {\bf Zbl.}0134.15705.


\bibitem{Gos}
{\sl Gosset Th.}
On the regular and semi-regular figures in space of $n$ dimensions~//
{\it Messenger Math.}, 1900, V.\,29. P.\,43--48.


\bibitem{Vinb85}
{\sl Vinberg E.~B.}
{\it Linear Representations of Groups}, Moscow: Nauka, 1985 (in Russian).
English translation: Springer Basel AG: 2010. Reprint of the 1989 Edition by Birkh\"auser-Verlag, Basel, {\bf MR}2761806, {\bf Zbl.}1206.20008.

\end{thebibliography}
\end{document}